\documentclass{amsart}
\usepackage{amsmath,amssymb,epsfig,amscd,amsthm,verbatim}
\usepackage[bookmarks,colorlinks,breaklinks]{hyperref}

\usepackage{verbatim}

\numberwithin{equation}{section}

\newtheorem{lemma}[equation]{Lemma}
\newtheorem{thm}[equation]{Theorem}

\newtheorem{prop}[equation]{Proposition}

\theoremstyle{remark}

\newtheorem*{acknowledgments}{Acknowledgments}

\newcommand{\lra}{\longrightarrow}

\newcommand{\bA}{{\mathbb A}}
\newcommand{\N}{{\mathbb N}}

\newcommand{\prep}{\operatorname{Prep}}

\newcommand{\C}{{\mathbb C}}

\newcommand{\bfa}{{\mathbf a}}

\newcommand{\bfh}{{\mathbf h}}
\newcommand{\bff}{{\mathbf f}}
\newcommand{\bfg}{{\mathbf g}}

  {\end{list}}

\begin{document}



\title[Simultaneously preperiodic points for families of polynomials]{Simultaneously preperiodic points for families of polynomials in normal form}

\author{Dragos Ghioca}
\address{
Dragos Ghioca\\
Department of Mathematics\\
University of British Columbia\\
Vancouver, BC V6T 1Z2\\
Canada
}
\email{dghioca@math.ubc.ca}

\author{Liang-Chung Hsia}
\address{Liang-Chung Hsia\\
Department of Mathematics\\
National Taiwan Normal University\\
Taipei, Taiwan, ROC
}
\email{hsia@math.ntnu.edu.tw}

\author{Khoa Dang Nguyen}
\address{
Khoa Nguyen\\
Department of Mathematics\\
University of British Columbia\\
Vancouver, BC V6T 1Z2\\
Canada
}
\email{dknguyen@math.ubc.ca}


\begin{abstract}
Let $d>m>1$ be integers, let $c_1,\dots, c_{m+1}$ be distinct complex numbers, and let $\bff(z):=z^d+t_1z^{m-1}+t_2z^{m-2}+\cdots + t_{m-1}z+t_m$ be an $m$-parameter family of polynomials. We prove that the set of $m$-tuples of parameters $(t_1,\dots, t_m)\in\C^m$ with the property that each $c_i$ (for $i=1,\dots, m+1$) is preperiodic under the action of the corresponding polynomial $\bff(z)$ is contained in finitely many hypersurfaces of the parameter space $\bA^m$.
\end{abstract}

\thanks{ The research of the first author was partially supported by NSERC. The second author was supported by MOST grant 105-2918-I-003-006. The third author was partially supported by a UBC-PIMS fellowship.}


\maketitle


\section{Introduction}

The principle of unlikely intersections for $1$-parameter families of rational functions $\bff_t$ predicts that given two starting points $c_1$ and $c_2$ which are not persistently preperiodic for the family $\bff$, if there exist infinitely many parameters $t$ such that  both $c_1$ and $c_2$ are preperiodic for $\bff_t$, then the two starting points are dynamically related; for more details, see \cite{Matt-Laura, Matt-Laura-2, Drinfeld, polynomials exact relation, rational functions, simultaneous preperiodic, DAO-1, DAO-2, M-Z-2, M-Z-3, M-Z-1}. For higher dimensional families of rational functions, there are very few definitive results, generally limited to $2$-parameter families of dynamical systems; see \cite[Theorem~1.4]{rational functions} and \cite[Theorem~1.4]{simultaneous preperiodic}.
In this paper we prove the following result regarding unlikely intersections for  arithmetic dynamics in higher dimensional parameter spaces.
\begin{thm}
\label{main result}
Let  $d>m> 1$ be integers, let $c_1,\dots, c_{m+1}\in\C$, and let
$$\bff(z):=z^{d}+ t_1z^{m-1}+\cdots + t_{m-1}z+t_{m}$$
be an $m$-parameter family of polynomials of degree $d$. For each point $\bfa = (a_1,\dots, a_m)$ of $\bA^m(\C)$ we let $\bff_{\bfa}$ be the corresponding polynomial defined over $\C$ obtained by specializing each $t_i$ to $a_i$ for $i=1,\dots, m$. Let $\prep(c_1, \ldots, c_{m+1})$ be the set consisting of parameters $\bfa\in \bA^m(\C)$ such that each starting point $c_i$ (for $i=1,\dots, m+1$) is preperiodic for $\bff_{\bfa}$. If the points $c_i$ are  distinct, then  $\prep(c_1,\ldots, c_{m+1})$ is not Zariski dense  in $\bA^m$.
\end{thm}

The polynomials $\bff(z)$ as in Theorem~\ref{main result} are in \emph{normal form}, i.e., they are monic of degree $d$ and the coefficient of $z^{d-1}$ is $0$. Since each polynomial $g$ is \emph{conjugate} with a polynomial in normal form, i.e., there exists a linear polynomial $\mu$ such that $\mu^{-1}\circ g \circ \mu$ is in normal form, one can focus on the dynamics corresponding to polynomials as in Theorem~\ref{main result}. In \cite[Theorem~1.4]{simultaneous preperiodic}, the special case $m=2$ in Theorem~\ref{main result} was proven, while the case of an arbitrary $m$ was conjectured in \cite[Question~1.1]{simultaneous preperiodic}. Our Theorem~\ref{main result} answers completely the problem raised in \cite{simultaneous preperiodic}.

If one considers $m$ (distinct) starting points $c_i$, then the set $\prep(c_1,\dots, c_m)$ is Zariski dense in $\bA^m$, as proven in \cite[Theorem~1.6]{Laura} (see also  \cite{portrait} for a discussion regarding  all possible preperiodicity portraits simultaneously realized for $m$ starting points by an $m$-parameter family of polynomials). On the other hand, there are numerous examples when the Zariski closure of $\prep(c_1,\dots, c_{m+1})$ is positive dimensional, and it may even have codimension $1$ in $\bA^m$ (see also  \cite[Introduction]{simultaneous preperiodic}). For example, if $m=3$, $d$ is even and $c_2=-c_1$ while $c_4=-c_3$, then the Zariski closure of $\prep(c_1,c_2,c_3,c_4)$ contains the plane $\mathcal{P}$ given by the equation $t_2=0$ in the parameter space $\bA^3$. Indeed, the specialization $$\bfg(z):=z^d+t_1z^2+t_3$$
of $\bff(z)=z^d+t_1z^2+t_2z+t_3$ along $\mathcal{P}$ yields a $2$-parameter family of even polynomials and due to the relations between the starting points $c_i$, we know that all $4$ starting points are preperiodic under the action of $\bfg$ if and only if $c_1$ and $c_3$ are preperiodic under the action of $\bfg$. Another application of \cite[Theorem~1.6]{Laura} yields that there exists a Zariski dense set of points $(t_1,t_3)\in\C^2$ such that both $c_1$ and $c_3$ are preperiodic for $\bfg$, thus proving that $\mathcal{P}$ is contained in the Zariski closure of $\prep(c_1,c_2,c_3,c_4)$.

Finally, we note that if $m=1$ in Theorem~\ref{main result}, then whenever $c_2=\zeta_d\cdot c_1$, for some $d$-th root of unity $\zeta_d$, we have that for each parameter $t$, the point $c_1$ is preperiodic under the action of $\bff(z)=z^d+t$ if and only if $c_2$ is preperiodic under the action of $\bff(z)$. In \cite[Theorem~1.1]{Matt-Laura}, it was shown that the above linear relation is also necessary so that there exist infinitely many parameters $t$ such that both $c_1$ and $c_2$ are preperiodic under the action of $z\mapsto z^d+t$. However, when $m>1$, there exists no linear automorphism of the entire family $\bff(z)$ (as opposed to the automorphism $z\mapsto \zeta_d\cdot z$ when $m=1$) and this allows us to prove Theorem~\ref{main result}.

We sketch now the plan for our paper. In Section~\ref{section useful} we state in Theorem~\ref{main useful} a key result proven in \cite{simultaneous preperiodic} for our problem. With the notation as in Theorem~\ref{main result}, assuming $\prep(c_1,\dots, c_{m+1})$ is Zariski dense in $\bA^m$, \cite[Theorem~5.1]{simultaneous preperiodic} yields that for each point $\bfa:=(a_1,\dots, a_m)\in\C^m$  in the parameter space, if $m$ of the starting points $c_i$ are preperiodic under the action of $\bff_\bfa$, then all $m+1$ starting points $c_i$ are preperiodic under the action of $\bff_\bfa$. Our strategy is to consider various lines $L\subset \bA^m$ along which each $c_i$ for $i=1,\dots, m-1$ is preperiodic under the action of $\bff$. Letting $\bfg_{t}$ be the $1$-parameter family of polynomials obtained by specializing $\bff$ along $L$ then \cite[Theorem~5.1]{simultaneous preperiodic} (coupled with \cite[Theorem~1.6]{Laura}) yields that there exist infinitely many parameters $t\in\C$ such that both $c_m$ and $c_{m+1}$ are preperiodic for $\bfg_t$. Then \cite[Theorem~1.3]{Matt-Laura-2} yields that the points $c_m$ and $c_{m+1}$ are dynamically related with respect to the family $\bfg_t$. In Section~\ref{section proof}, using an in-depth analysis of this information for two different lines $L$, we derive a contradiction, thus proving Theorem~\ref{main result}. It is interesting to note that this strategy works as long as $m>2$; however, we note that the case $m=2$ was proven in \cite[Theorem~1.4]{simultaneous preperiodic} using a similar strategy, but this time extracting slightly different information from using a single line $L$ in the parameter plane $\bA^2$ along which $c_1$ is fixed.

\begin{acknowledgments}
We are indebted to Tom Tucker for many enlightening conversations over the past 6 years on the topic of unlikely intersections in arithmetic dynamics. This paper was written when the second author visited UBC and PIMS. He would like to thank UBC and PIMS and he thanks the first author for his hospitality and help during the visit.
\end{acknowledgments}


\section{Useful results}
\label{section useful}

We start by recalling the traditional assumption from algebraic dynamics that for a polynomial $f$ and a positive integer $n$, we denote by $f^n=f\circ \cdots \circ f$ its composition with itself $n$ times; furthermore, $f^0$ always denotes the identity function. A point $a$ is called \emph{preperiodic} under the action of $f$ if its forward orbit under $f$ consist of only finitely many distinct elements, i.e., there exist integers $n>m\ge 0$ such that $f^n(a)=f^m(a)$. Also, as a matter of notation, $\N$ denotes the set of all positive integers, while $\N_0:=\N\cup\{0\}$.

It will be useful for our proof of Theorem~\ref{main result} to know all polynomials commuting with an iterate of a given polynomial. Before stating \cite[Theorem~2.3]{Nguyen}, we recall first the definition of the \emph{$d$-th Chebyshev polynomial $T_d(z)$} (for some integer $d\ge 2$), i.e., the unique polynomial satisfying the identity $T_d(z+1/z)=z^d+1/z^d$ for all $z$. We have
\cite[Theorem~2.3]{Nguyen}:

\begin{thm}
\label{Khoa's theorem}
Let $K$ be an algebraically closed field of characteristic $0$, let $d\ge 2$ be an integer, and let $g(z)\in K[z]$ be a polynomial of degree $d>1$ which is not conjugate to $z^d$ or to $\pm T_d(z)$.
\begin{itemize}
\item [(a)] If $h(z)\in K[z]$ has degree at least $2$ such that $h$ commutes with an iterate of $g$, i.e., $h\circ g^{n}=g^n\circ h$ for some $n\in\N$, then $h$ and $g$ have a common iterate.
\item [(b)] Let $M(g^\infty)$ denote the collection of all linear  polynomials commuting with an iterate of $g$. Then $M(g^\infty)$ is a finite cyclic group under composition.
\item [(c)] Let $\tilde{g}(z)\in K[z]$ be a polynomial of minimum degree $\tilde{d}\ge 2$ such that $\tilde{g}$ commutes with an iterate of $g$. Then there exists $D=D_g>0$ relatively prime to the order of $M(g^\infty)$ such that  $\tilde{g}\circ L=L^D\circ \tilde{g}$ for every $L\in M(g^\infty)$.
\item [(d)] $\left\{\tilde{g}^m\circ L\colon  m\in\N_0\text{ and }L\in M(f^\infty)\right\}=\left\{L\circ\tilde{g}^m\colon m\in\N_0\text{ and } L\in M(f^\infty)\right\}$, and this set describes exactly all polynomials $h$ commuting with an iterate of $g$.		
\end{itemize}
\end{thm}

We state now the key result (proven in \cite[Theorem~5.1]{simultaneous preperiodic}) which we will use for deriving the conclusion in Theorem~\ref{main result}.
\begin{thm}[\cite{simultaneous preperiodic}]
\label{main useful}
Let  $d>m> 1$ be integers, let $c_1,\dots, c_{m+1}$ be distinct complex numbers, and let $\bff(z):=z^{d}+ t_1z^{m-1}+\cdots + t_{m-1}z+t_{m}$
be an $m$-parameter family of polynomials of degree $d$. For each point $\bfa = (a_1,\dots, a_m)$ of $\bA^m(\C)$ we let $\bff_{\bfa}$ be the corresponding polynomial defined over $\C$ obtained by specializing each $t_i$ to $a_i$ for $i=1,\dots, m$. Let $\prep(c_1, \ldots, c_{m+1})$ be the set consisting of parameters $\bfa\in \bA^m(\C)$ such that each starting point $c_i$ (for $i=1,\dots, m+1$) is preperiodic for $\bff_{\bfa}$. Assume  $\prep(c_1,\ldots, c_{m+1})$ is Zariski dense  in $\bA^m$. Then for each $\bfa\in\C^m$ such that $c_1,\dots, c_m$ are preperiodic for $\bff_\bfa$, we have that also $c_{m+1}$ is preperiodic for $\bff_\bfa$.
\end{thm}

We let $L$ be a line in the parameter space $\bA^m$  parametrized with respect to the coordinates $(t_1,\dots, t_m)$ of $\bA^m$, as follows:
$$t_1:=t\text{ and }t_i=\alpha_it+\beta_i\text{ for }i=2,\dots, m,$$
for some complex numbers $\alpha_i,\beta_i$. Furthermore, we assume
\begin{equation}
\label{alpha 2 nonzero 0}
\alpha_2\ne 0.
\end{equation}
We let $\bfg:=\bfg_t$ be the specialization of $\bff$ along the line $L$, i.e.,
\begin{equation}
\label{form of g}
\bfg_t(z)=z^d+tz^{m-1}+\sum_{i=2}^m (\alpha_i t+\beta_i)z^{m-i}.
\end{equation}

The next result is essential for the proof of Theorem~\ref{main result}.
\begin{prop}
\label{what commutes with g}
Let $K=\overline{\C(t)}$, and let $\bfh[z]\in K[z]$. With the above notation \eqref{form of g} for $\bfg$, if $\bfh$ commutes with an iterate of $\bfg$ then $\bfh=\bfg^\ell$ for some $\ell\in\N_0$.
\end{prop}

\begin{proof}
The desired conclusion follows from the next three lemmas coupled with Theorem~\ref{Khoa's theorem} describing all polynomials commuting with an iterate of a given polynomial.

\begin{lemma}
\label{lemma 0}
With the above notation, $\bfg(z)$ is not conjugate to $z^d$ or to $\pm T_d(z)$.
\end{lemma}

\begin{proof}[Proof of Lemma~\ref{lemma 0}.]
Since $z^d$, $T_d(z)$ and also $\bfg(z)$ are polynomials in normal form, then assuming that for some linear polynomial $\mu$, we have that $\mu^{-1}\circ \bfg\circ \mu$ is either $z^d$ or $\pm T_d(z)$, we conclude that $\mu(z)=\zeta\cdot z$ for some root of unity $\zeta$. Indeed, letting $\mu(z)=az+b$, we get first that $b=0$ since $\bfg(z)$, $z^d$ and $\pm T_d(z)$ have coefficient equal to $0$ for their monomial of degree $d-1$. Then equating the leading coefficient in each of the above polynomials  yields that $a$ must be a root of unity.  Because $z^d$ and $\pm T_d(z)$ have constant coefficients, i.e., there is no dependence on $t$, we conclude that $\bfg$ is not conjugate to a monomial or $\pm$Chebyshev polynomial.
\end{proof}

\begin{lemma}
\label{lemma 1}
If $\mu(z)$ is a linear polynomial commuting with an iterate of $\bfg$, then $\mu(z)=z$ for all $z$.
\end{lemma}

\begin{proof}[Proof of Lemma~\ref{lemma 1}.]
We let $\mu(z)=az+b$ and assume $\mu\circ \bfg^n=\bfg^n\circ \mu$ for some $n\in\N$.
Again using the fact that $\bfg$ (and thus also $\bfg^n$) is in normal form, we conclude that $b=0$. Then using the fact that $\bfg^n(z)$ has nonzero terms of degrees $d^n-d+m-1$ and $d^n-d+m-2$ (using \eqref{form of g} and \eqref{alpha 2 nonzero 0} along with an easy induction on $n$), we conclude that $1=a^{d^n-d+m-2}=a^{d^n-d+m-3}$; hence $a=1$, as claimed.
\end{proof}

\begin{lemma}
\label{lemma 2}
There is no polynomial $\bfh_1(z)\in K[z]$ and no integer $e>1$ such that $\bfh_1^e=\bfg$.
\end{lemma}

\begin{proof}[Proof of Lemma~\ref{lemma 2}.]
We argue by contradiction and therefore assume $\bfh_1^e=\bfg$ with some integer $e>1$ and some polynomial $\bfh_1\in K[z]$ of degree $s>1$; furthermore, we  assume $\bfh_1$ has minimal degree among all such polynomials. According to Theorem~\ref{Khoa's theorem}~part~(d) along with Lemmas~\ref{lemma 0}~and~\ref{lemma 1}, we know that all polynomials commuting with $\bfg$ are of the form $\bfh_1^n$ for some $n\in\N_0$.

First, we claim that $\bfh_1(z)\in \C(t)[z]$. Indeed, otherwise there exists some Galois automorphism $\tau$ of $K$ fixing $\C(t)$ such that $\bfh_2:=(\bfh_1)^\tau\ne \bfh_1$ (i.e., some coefficient of $\bfh_1$ is not fixed by $\tau$). But then also $\bfh_2^e=\bfg$ (since each coefficient of $\bfg$ is fixed by $\tau$) and therefore $\bfh_2=\bfh_1$ since they both have same degree and commute with $\bfg$. This is a contradiction and so, $\bfh_1(z)\in \C(t)[z]$.

Second, we claim that $\bfh_1\in \C[t][z]$. Because $\bfh_1^e=\bfg$, then we know that $\bfh_1(z)=\sum_{i=0}^s a_iz^i$ for some $a_i\in \C(t)$; since $\bfg(z)$ is monic, we have that $a_s$ is a root of unity. Now, assuming $i_1\in\{0,\dots, s-1\}$ is the largest integer such that $a_{i_1}\notin \C[t]$, an easy induction on $e$ yields that the coefficient of $z^{s(e-1)+i_1}$ in $\bfh_1^e=\bfg$ is not contained in $\C[t]$, which is contradiction.

So, we know that $\bfh_1(z)\in\C[t][z]$.
Since $\bfg$ is in normal form, we conclude that $\bfh_1$ must have no nonzero term of degree $s-1$. Now, let $D$ be the maximum degree in $t$ of the coefficients of $\bfh_1$; clearly, $D\ge 1$ since $\bfg_{t}$ is not a constant family in $t$. Then for all but finitely many $c\in\C$, the degree in $t$ of $\bfh_1(c)$ equals $D$; let $c$ be one such complex number. An easy computation (using the fact that $\bfh_1(z)$ has no terms of degree $s-1$) yields that the degree in $t$ of $\bfh_1^e(c)$ equals $Ds^{e-1}$. On the other hand, the degree in $t$ of $\bfg(c)$ is at most $1$. So, the assumption that $e>1$ yields a contradiction, thus concluding the proof of Lemma~\ref{lemma 2}.
\end{proof}

Lemmas~\ref{lemma 0}, \ref{lemma 1} and \ref{lemma 2}, along with Theorem~\ref{Khoa's theorem} finish the proof of Proposition~\ref{what commutes with g}.
\end{proof}


\section{Proof of our main result}
\label{section proof}

\begin{proof}[Proof of Theorem~\ref{main result}.]
Since the case $m=2$ was proven in \cite[Theorem~1.4]{simultaneous preperiodic}, we assume from now on that $m>2$. Also, we proceed by contradiction, i.e., we assume that the set $\prep(c_1,\dots, c_{m+1})$ is Zariski dense in $\bA^m$; this allows us to apply Theorem~\ref{main useful}.

Now, since the numbers $c_i$ are distinct, then clearly we can find $(m-1)$ of them whose sum is nonzero; so, without loss of generality, we assume that
\begin{equation}
\label{sum of m-1 is nonzero}
\sum_{i=1}^{m-1}c_i \ne 0.
\end{equation}

For each function (not necessarily injective) $\sigma:\{1,\dots, m-1\}\lra \{1,\dots, m-1\}$, we let $L_\sigma\subset \bA^m$ be the line in the parameter space along which the following relations hold:
\begin{equation}
\label{preperiodic line sigma}
\bff(c_i)=c_{\sigma(i)}\text{ for each }i=1,\dots, m-1.
\end{equation}
Indeed, in order to solve the system of equations \eqref{preperiodic line sigma} in the variables $t_i$, we let $t_1:=t$ and then solve each of the $t_i$'s (for $i=2,\dots, m$) in terms of the variable $t$, and in each case we get that $t_i$ is a polynomial $T_{\sigma,i}(t)$ of degree at most $1$. The fact that the system \eqref{preperiodic line sigma} is solvable follows from Cramer's Rule using the fact that the coefficients matrix is an invertible Vandermonde matrix since $c_i\ne c_j$ if $1\le i<j\le m-1$.

Thus, the points $c_i$ (for $i=1,\dots, m-1$) are preperiodic along $L_\sigma$; we let $\bfg_{\sigma}=\bfg_{\sigma,t}$ the specialization of $\bff$ along the line $L_\sigma$. Furthermore,  there exist polynomials $A,B_\sigma\in\C[z]$ such that
\begin{equation}
\label{formula g 1}
\bfg_{\sigma,t}(z)=A(z)t+B_\sigma(z).
\end{equation}
A simple computation (using the fact that $A(z)$ is a monic polynomial of degree  $m-1$ and that $\bfg_{\sigma,t}(c_i)=A(c_i)t+B(c_i)$ is a constant polynomial in $t$ for each $i=1,\dots, m-1$) yields that
\begin{equation}
\label{what is A}
A(z)=\prod_{i=1}^{m-1}(z-c_i),
\end{equation}
which confirms the fact that $A(z)$ is independent of the choice of the function $\sigma$. So, there exist some complex numbers $\alpha_i$ and $\beta_{\sigma,i}$ (for $i=2,\dots, m$) such that
\begin{equation}
\label{form of g 6}
\bfg_{\sigma,t}(z)=z^d+tz^{m-1}+(\alpha_2t+\beta_{\sigma,2})z^{m-2}+\cdots + (\alpha_{m-1}t+\beta_{\sigma,m-1})z+\alpha_mt + \beta_{\sigma,m}
\end{equation}
Furthermore, according to \eqref{what is A}, we have that
\begin{equation}
\label{alpha 2 nonzero}
\alpha_2=-\sum_{i=1}^{m-1}c_i\ne 0.
\end{equation}

Equation \eqref{what is A} yields that for any $c\notin\{ c_1,\dots, c_{m-1}\}$,  we have that $\deg_t(\bfg_{\sigma,t}(c))=1$ and furthermore (by induction), for any $n\ge 1$, we have that
\begin{equation}
\label{degree of iterate g in t}
\deg_t\left(\bfg_{\sigma,t}^n(c)\right)=d^{n-1}.
\end{equation}
Because the points $c_i$ (for $i=1,\dots, m-1$) are persistently preperiodic for $\bfg_{\sigma,t}$, Theorem~\ref{main useful} yields that for each parameter $t\in\C$, we have that $c_m$ is preperiodic for $\bfg_{\sigma,t}$ if and only if $c_{m+1}$ is preperiodic for $\bfg_{\sigma,t}$. Then  \cite[Theorem~1.3]{Matt-Laura-2} yields that there exists some polynomial $\bfh(z) = \bfh_\sigma(x)\in\C[t][z]$ commuting with an iterate of $\bfg_\sigma$ and there exist positive integers $n_m,n_{m+1}$ such that $\bfg_\sigma^{n_m}(c_m)=\bfh\left(\bfg_\sigma^{n_{m+1}}(c_{m+1})\right)$. Proposition~\ref{what commutes with g} (see also \eqref{alpha 2 nonzero}) allows us to assume that $\bfh$ is the identity. Furthermore, using \eqref{degree of iterate g in t}, we conclude that $n_m=n_{m+1}=:n$. Next we prove that we may assume that $n=2$.

\begin{prop}
\label{n is 2}
Let $n$ be an integer larger than $2$. If $\bfg_{\sigma,t}^n(c_m)=\bfg_{\sigma,t}^n(c_{m+1})$, then $\bfg_{\sigma,t}^{n-1}(c_m)=\bfg_{\sigma,t}^{n-1}(c_{m+1})$.
\end{prop}

\begin{proof}[Proof of Proposition~\ref{n is 2}.]
First we prove that there exists some $d$-th root of unity $\zeta$ such that $\bfg_{\sigma}^{n-1}(c_m)=\zeta\cdot \bfg_{\sigma}^{n-1}(c_{m+1})$, and then we will prove that actually $\zeta=1$.

Using \eqref{degree of iterate g in t}, we have that, as a polynomial in $t$,
\begin{align}
\nonumber
\bfg_{\sigma,t}^n(c_m)\\
\nonumber
& = \left(\bfg_{\sigma,t}^{n-1}(c_m)\right)^d + t\left(\bfg_{\sigma,t}^{n-1}(c_m)\right)^{m-1} + \sum_{i=2}^m (\alpha_it+\beta_{\sigma,i})\cdot \left(\bfg_{\sigma,t}^{n-1}(c_m)\right)^{m-i} \\
\label{dominant term 1}
& =   \left(\bfg_{\sigma,t}^{n-1}(c_m)\right)^d + O\left(t^{d^{n-2}(m-1)+1}\right),
\end{align}
where the big-$O$ term from \eqref{dominant term 1} denotes the fact that the remaining powers of $t$ from the expansion of $\bfg_{\sigma,t}^n(c_m)$ have degree bounded by $d^{n-2}(m-1)+1$. A similar formula holds for $\bfg_{\sigma,t}^n(c_{m+1})$. Therefore, the equality $\bfg_{\sigma,t}^n(c_m)=\bfg_{\sigma,t}^n(c_{m+1})$ yields that
\begin{equation}
\label{dominant term 11}
\deg_t\left(\bfg_{\sigma,t}^{n-1}(c_m)^d-\bfg_{\sigma,t}^{n-1}(c_{m+1})^d\right)\le d^{n-2}(m-1)+1.
\end{equation}
Now, let $\zeta_d$ be a primitive $d$-th root of unity. If there is no $i\in\{0,\dots, d-1\}$ such that $\bfg_{\sigma,t}^{n-1}(c_m)=\zeta_d^i\cdot \bfg_{\sigma,t}^{n-1}(c_{m+1})$, then
\begin{equation}
\label{dominant term 2}
\left(\bfg_{\sigma,t}^{n-1}(c_m)\right)^d - \left(\bfg_{\sigma,t}^{n-1}(c_{m+1})\right)^d = \prod_{i=0}^{d-1}\left(\bfg_{\sigma,t}^{n-1}(c_m) - \zeta_d^i\cdot \bfg_{\sigma,t}(c_{m+1})^{n-1}\right)
\end{equation}
is a polynomial in $t$ of degree at least $\deg_t \left(\bfg_{\sigma,t}(c_m)\right)^{d-1}=d^{(n-2)}\cdot (d-1)$ since at most one of the terms from the product appearing in \eqref{dominant term 2} may have degree less than $\deg_t\left(\bfg_{\sigma,t}^{n-1}(c_m)\right)$. This contradicts \eqref{dominant term 11} (note that $n>2$), thus proving that one of the terms in the product appearing in \eqref{dominant term 2} must be $0$ and so, there exists a root of unity $\zeta=\zeta_d^{i_0}$ (for some $i_0=0,\dots, d-1$) such that
\begin{equation}
\label{dominant term 33}
\bfg_{\sigma,t}^{n-1}(c_m)=\zeta\cdot \bfg_{\sigma,t}^{n-1}(c_{m+1}).
\end{equation}
Next we prove that $\zeta=1$ (i.e., $i_0=0$). For this we need to refine the expansion from \eqref{dominant term 1}, as follows:
\begin{align}
\nonumber
& \bfg_{\sigma,t}^n(c_m)\\
\label{dominant term 3}
& = \left(\bfg_{\sigma,t}^{n-1}(c_m)\right)^d + t\cdot \left(\bfg_{\sigma,t}^{n-1}(c_m)\right)^{m-1}\\
\nonumber
& + (\alpha_2t+\beta_{\sigma,2})\cdot \left(\bfg_{\sigma,t}^{n-1}(c_m)\right)^{m-2} + O\left(t^{d^{n-2}(m-3)+1}\right)
\end{align}
and similarly, using \eqref{dominant term 33}, we get
\begin{align}
\nonumber
& \bfg_{\sigma,t}^n(c_{m+1})\\
\nonumber
& =\left(\bfg_{\sigma,t}^{n-1}(c_{m+1})\right)^d + t\cdot \left(\bfg_{\sigma,t}^{n-1}(c_{m+1})\right)^{m-1}\\
\nonumber
& + (\alpha_2t+\beta_{\sigma,2})\cdot \left(\bfg_{\sigma,t}^{n-1}(c_{m+1})\right)^{m-2} + O\left(t^{d^{n-2}(m-3)+1}\right)\\
\label{dominant term 32}
&  = \left(\bfg_{\sigma,t}^{n-1}(c_{m})\right)^d + t\cdot \zeta^{m-1} \left(\bfg_{\sigma,t}^{n-1}(c_{m})\right)^{m-1} \\
\nonumber
& + (\alpha_2t+\beta_2)\cdot \zeta^{m-2}\left(\bfg_{\sigma,t}^{n-1}(c_{m})\right)^{m-2} + O\left(t^{d^{n-2}(m-3)+1}\right).
\end{align}
The equality $\bfg_{\sigma,t}^n(c_m)=\bfg_{\sigma,t}^n(c_{m+1})$ coupled with expansions \eqref{dominant term 3} and \eqref{dominant term 32} yields first that $\zeta^{m-1}=1$, and then re-using \eqref{dominant term 3} and \eqref{dominant term 32} yields that $\zeta^{m-2}=1$. So, $\zeta=1$, as desired.
\end{proof}

So, we know that $\bfg_{\sigma,t}^2(c_m)=\bfg_{\sigma,t}^2(c_{m+1})$. Using \eqref{formula g 1}, we get that
\begin{align}
\nonumber
& 0=\bfg_{\sigma,t}^2(c_m)-\bfg_{\sigma,t}^2(c_{m+1}) \\
\label{dominant term 4}
& = \left(A(c_m)t+B_{\sigma}(c_m)\right)^d - \left(A(c_{m+1})t+B_{\sigma}(c_{m+1})\right)^d \\
\nonumber
& + t\left(A(c_m)t+B_{\sigma}(c_m)\right)^{m-1} - t\left(A(c_m)t+B_{\sigma}(c_m)\right)^{m-1} + O\left(t^{m-1}\right).
\end{align}
Comparing the terms of degree $d$ we get
\begin{equation}
\label{dominant term 5}
A(c_{m+1})=\xi\cdot A(c_m),
\end{equation}
for some $\xi\in\C$ such that $\xi^d=1$. Note that $\xi$ is independent of $\sigma$, since $A(z)$ is independent of $\sigma$.

\begin{prop}
\label{sigma doesn't matter}
The quantity $B_{\sigma}(c_{m+1})-\xi B_{\sigma}(c_m)$ is independent of the function $\sigma$.
\end{prop}

\begin{proof}[Proof of Proposition~\ref{sigma doesn't matter}.]
Our analysis splits into two cases: either $m<d-1$, or $m=d-1$.

If $m<d-1$, then comparing the coefficient of $t^{d-1}$ in \eqref{dominant term 4}, we get $A(c_m)^{d-1}B_\sigma(c_m)=A(c_{m+1})^{d-1}B_\sigma(c_{m+1})$ and so, \eqref{dominant term 5} yields that $B_{\sigma}(c_{m+1})=\xi\cdot B_{\sigma}(c_m)$ (note that $A(c_m)\ne 0$, according to \eqref{what is A}), thus providing the desired conclusion.

If $m=d-1$, then again comparing the coefficient of $t^{d-1}$ in \eqref{dominant term 4} yields that
\begin{equation}
\label{dominant term 6}
0=dA(c_m)^{d-1}B_{\sigma}(c_m)-dA(c_{m+1})^{d-1}B_{\sigma}(c_{m+1}) + A(c_m)^{d-2}-A(c_{m+1})^{d-2}.
\end{equation}
Using \eqref{dominant term 5} and \eqref{dominant term 6}, we obtain that
$$B_\sigma(c_{m+1})-\xi B_\sigma(c_m) = \frac{\xi-\xi^{-1}}{dA(c_m)}.$$
This concludes the proof of Proposition~\ref{sigma doesn't matter}.
\end{proof}

Using Lagrange interpolation for the polynomial $B_{\sigma}(z)-z^d$ which has degree at most $m-2$, one computes that
\begin{equation}
\label{what is B}
B_\sigma(z)=z^d+\sum_{i=1}^{m-1} \left(c_{\sigma(i)}-c_i^d\right) \cdot \frac{A(z)}{(z-c_i)\cdot A'(c_i)},
\end{equation}
where $A'(z)$ is the derivative of the polynomial $A(z)$.
Next we will consider two special functions $\sigma$: one of them is the identity function $\sigma_1$ which maps $c_i$ to $c_i$ for each $i=1,\dots, m-1$, while the second function $\sigma_2$ differs from $\sigma_1$ only when evaluated at $c_1$, i.e,
$$\sigma_2(c_1)=c_2\text{ and }\sigma_2(c_i)=c_i\text{ for }i=2,\dots, m-1.$$
Proposition~\ref{sigma doesn't matter} yields that
\begin{equation}
\label{dominant term 7}
B_{\sigma_2}(c_{m+1})-B_{\sigma_1}(c_{m+1})=\xi\cdot \left(B_{\sigma_2}(c_{m+1})-B_{\sigma_1}(c_m)\right).
\end{equation}
Using \eqref{what is B} along with \eqref{dominant term 7} yields that
\begin{align*}
  0 & =\frac{(c_2-c_1)A(c_{m+1})}{A'(c_1)(c_{m+1}-c_1)} - \frac{(c_2-c_1)\cdot \xi A(c_m)}{A'(c_1)(c_m-c_1)} \\
  & = \frac{(c_2-c_1)A(c_{m+1})}{A'(c_1)}\left(\frac{1}{(c_{m+1}-c_1)}-\frac{1}{(c_{m}-c_1)}\right) \quad \text{since $A(c_{m+1}) = \xi A(c_m)$}\\
& = \frac{(c_2-c_1)A(c_{m+1})}{A'(c_1)}\cdot \frac{c_m-c_{m+1}}{(c_{m+1}-c_1)(c_m-c_1)}
\end{align*}
Therefore, either  $c_{m+1}=c_m$, or $A(c_{m+1})=0$, i.e., $c_{m+1}=c_i$ for some $i=1,\dots, m-1$; this contradicts the fact that the starting points $c_i$ are all distinct. In conclusion, $\prep(c_1,\dots, c_{m+1})$ is contained in finitely many hypersurfaces of  $\bA^m$.
\end{proof}





\end{document}